\documentclass[11pt]{amsart}
\usepackage{amsmath,amsthm,amssymb}
\usepackage{fullpage}
\usepackage{hyperref}
\usepackage{verbatim}
\usepackage{xcolor}
\usepackage[all]{xy}
  \SelectTips{cm}{10}
  \everyxy={<2.5em,0em>:} 
\usepackage{tikz}
\usepackage{picinpar} 

\renewcommand{\theequation}{\thesection.\arabic{equation}}
\makeatletter\@addtoreset{equation}{section}\makeatother

\theoremstyle{plain}
  \newtheorem{theorem}[equation]{Theorem}
  \newtheorem*{theorem*}{Theorem}
  \newtheorem{proposition}[equation]{Proposition}

\theoremstyle{definition}
  \newtheorem{definition}[equation]{Definition}
  \newtheorem{lemma}[equation]{Lemma}
  \newtheorem{corollary}[equation]{Corollary}

\theoremstyle{remark}
  
  \newtheorem{remark}[equation]{Remark}

 \DeclareFontFamily{U}{manual}{}
 \DeclareFontShape{U}{manual}{m}{n}{ <->  manfnt }{}
 \newcommand{\manfntsymbol}[1]{%
    {\fontencoding{U}\fontfamily{manual}\selectfont\symbol{#1}}}

\makeatletter
    {\hspace*{\fill}
     \endgraf\endgroup\end{trivlist}}
 \newenvironment{example}[1][]{
   \refstepcounter{equation}
   \begin{proof}[Example~\theequation%
   \@ifnotempty{#1}{ (#1)}.]
   }
  {\end{proof}}
\makeatother

 \DeclareFontFamily{OT1}{pzc}{}
 \DeclareFontShape{OT1}{pzc}{m}{it}{<-> s * [1.100] pzcmi7t}{}
 \DeclareMathAlphabet{\mathpzc}{OT1}{pzc}{m}{it}

\newcommand{\hhat}[1]{\widehat{#1}}

\DeclareMathOperator{\colim}{colim}

\renewcommand{\AA}{\mathbb A}

\newcommand{\gp}{\mathrm{gp}}
\newcommand{\anton}[1]{{\color{red}[[\ensuremath{\bigstar\bigstar\bigstar} #1]]}}

\newcommand{\<}{\langle}
\renewcommand{\>}{\rangle} 
\newcommand{\m}{\mathfrak m}
\newcommand{\D}{\mathcal D}
\renewcommand{\L}{\mathcal L}
\newcommand{\E}{\mathcal E}

\newcommand{\U}{\mathcal U}
\newcommand{\I}{\mathcal I}

\newcommand{\X}{\mathcal X}
\newcommand{\Y}{\mathcal Y}
\newcommand{\Z}{\mathcal Z}
\newcommand{\ZZ}{\mathbb Z}

\newcommand{\QQ}{\mathbb Q}

\newcommand{\GG}{\mathbb G}

\newcommand{\bbar}[1]{\overline{#1}}
\newcommand{\ttilde}[1]{\widetilde{#1}}

\renewcommand{\setminus}{\smallsetminus}

\let\hom\relax
\DeclareMathOperator{\hom}{Hom}
\renewcommand{\O}{\mathcal{O}}
\DeclareMathOperator{\aut}{Aut}

\newcommand{\Ga}{\Gamma}
\DeclareMathOperator{\spec}{Spec}

\DeclareMathOperator{\sym}{Sym}
\newcommand{\smat}[1]{\left(\begin{smallmatrix}#1\end{smallmatrix}\right)}
\DeclareMathOperator{\stab}{Stab}
\newcommand{\F}{\mathcal F}

\newcommand{\g}{\mathfrak g}

 \def\ari[#1]{\ar@{^(->}[#1]}
 \def\are[#1]{\ar[#1]^{\txt{\'et}}}
 \def\areh[#1]{\ar[#1]|{\txt{$H$-eq}}^{\txt{\'et}}}
 \def\ars[#1]{\ar@{->>}[#1]}
 \newcommand{\dplus}{\ar@{}[d]|{\mbox{$\oplus$}}}
 \newcommand{\dtimes}{\ar@{}[d]|{\mbox{$\times$}}}

\begin{document}
\title{Toric Stacks II: Intrinsic Characterization of Toric Stacks}
\author{Anton Geraschenko}
\author{Matthew Satriano}
\thanks{The second author was partially supported by NSF grant DMS-0943832.}
\date{}

\begin{abstract}
 The purpose of this paper and its prequel \cite{toricartin1} is to introduce and develop a theory of toric stacks which encompasses and extends the notions of toric stacks defined in \cite{lafforgue,bcs,fmn,iwanari,can,tyomkin}, as well as classical toric varieties.

 While the focus of the prequel \cite{toricartin1} is on how to work with toric stacks, the focus of this paper is how to show a stack is toric. For toric varieties, a classical result says that a finite type scheme with an action of a dense open torus arises from a fan if and only if it is normal and separated. In \cite[Theorem 7.24]{fmn} and \cite[Theorem 1.3]{iwanari}, it is shown that a smooth separated DM stack with an action of a dense open torus arises from a stacky fan. In the same spirit, the main result of this paper is that any Artin stack with an action of a dense open torus arises from a stacky fan under reasonable hypotheses.
\end{abstract}
\maketitle

\tableofcontents

\section{Introduction}\label{sec:introduction}
This paper, together with its prequel \cite{toricartin1}, introduces a theory of toric stacks which encompasses and extends the many pre-existing theories in the literature \cite{lafforgue,bcs,fmn,iwanari,can,tyomkin}. Recall from \cite[Definition 1.1]{toricartin1} that a \emph{toric stack} is defined to be the stack quotient $[X/G]$ of a normal toric variety $X$ by a subgroup $G$ of the torus of $X$. As with toric varieties, one can understand toric stacks through a combinatorial theory of \emph{stacky fans}. In Toric Stacks I \cite{toricartin1}, we introduce the notion of stacky fan, show that every toric stack comes from a stacky fan, and develop a rich dictionary between stacky fans and their associated toric stacks, thereby allowing one to easily read off properties of a toric stack from its stacky fan.

In contrast to \cite{toricartin1}, which develops the tools to study a toric stack, the focus of this paper is how to show that a given stack is toric in the first place. A classical result (see for example \cite[Corollary 3.1.8]{cls}) shows that if $X$ is a finite type \emph{scheme} with a dense open torus $T$ whose action on itself extends to $X$, then $X$ is a toric variety if and only if it is normal and separated. Similarly, the main result of this paper states that with suitable hypotheses, if $\X$ is an Artin stack with a dense open torus $T$ whose action on itself extends to an action of $\X$, then $\X$ is a toric stack:

{\renewcommand{\theequation}{\ref{thm:main}}
\begin{theorem}
 Let $\X$ be an Artin stack of finite type over an algebraically closed field $k$ of characteristic $0$. Suppose $\X$ has an action of a torus $T$ and a dense open substack which is $T$-equivariantly isomorphic to $T$. Then $\X$ is a toric stack if and only if the following conditions hold:
 \begin{enumerate}
  \item $\X$ is normal,
  \item $\X$ has affine diagonal,
  \item geometric points of $\X$ have linearly reductive stabilizers, and
  \item \label{intro:global-type} every point of $[\X/T]$ is in the image of an \'etale representable map from a stack of the form $[U/G]$, where $U$ is quasi-affine and $G$ is an affine group. (See Definition \ref{def:global-type} and Remark \ref{rmk:everything-is-global-type}; a forthcoming result of Alper, Hall, and Rydh shows that this condition is superfluous.)
 \end{enumerate}
\end{theorem}
\addtocounter{equation}{-1}}

Note that unlike the classical result about toric varieties,
we cannot require our stacks to be separated. Indeed, algebraic stacks which are not Deligne-Mumford are hardly ever separated. The condition that the stack have affine diagonal essentially replaces the separatedness condition (see Remark \ref{rmk:affine-diagonal<->separated}). In particular, there exist ``toric schemes'': toric stacks which are schemes, but which are not toric varieties because they are not separated (see Example \ref{Eg:non-affine-diag}).

A forthcoming result of Alper, Hall, and Rydh implies that condition (\ref{intro:global-type}) follows from the other hypotheses and can therefore be removed (see Remark \ref{rmk:everything-is-global-type}(\ref{alper-hall-rydh})). Given that result, we recover new proofs (in the case of trivial generic stabilizer) of \cite[Theorem 1.3]{iwanari} and \cite[Theorem 7.24]{fmn}, which establish the analogous result for toric stacks which are smooth, separated, and Deligne-Mumford.

Notably, the results in \cite{iwanari} and \cite{fmn} impose the hypothesis that the stack has a coarse space which is a scheme. In contrast, the characterization given by Theorem \ref{thm:main} does not assume that $\X$ has a coarse (or good) moduli space \emph{at all}. As a corollary, we see that any toric algebraic space satisfying the conditions of the theorem is in fact a scheme (see Remark \ref{rmk:no-toric-spaces} and Example \ref{eg:a-toric-space}).

\begin{remark}
The techniques in this paper work over any separably closed field, but we work over an algebraically closed field $k$ of characteristic zero to avoid confusing hypotheses (e.g.~that every group we consider is smooth).
\end{remark}

\subsection*{Logical Dependence of Sections}

The logical dependence of sections is roughly as follows:
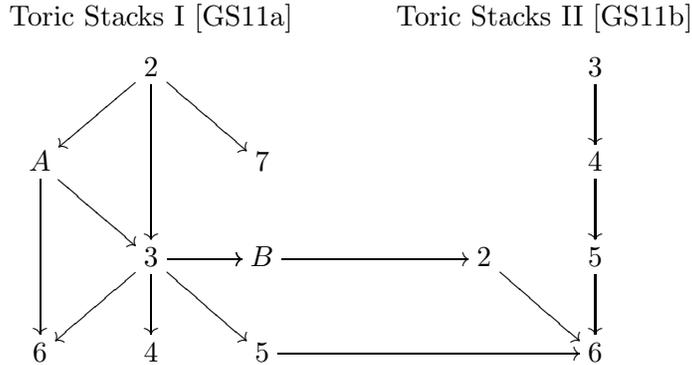
\[\xymatrix@!0 @C+1.5pc @R+1pc{
   & {2}\ar[dd]\ar[dl]\ar[dr] \POS p+(0,.7) *\txt{Toric Stacks I \cite{toricartin1}} & & & & {\ref{sec:technical}}\ar[d] \POS p+(-.7,.7) *\txt{Toric Stacks II \cite{toricartin2}}\\
 {A}\ar[dd]\ar[dr] & & {7} & & & {\ref{sec:local-structure}}\ar[d]\\
 & {3}\ar[dl]\ar[d]\ar[dr]\ar[r]\ar[r] & {B}\ar[rr] & & {\ref{sec:local-construction}} \ar[dr] & {\ref{sec:main-smooth}} \ar[d]
 \\
 {6} & {4} &  {5}\ar[rrr] & & & {\ref{sec:main}}
}\]

\subsection*{Acknowledgments}
We thank Jesse Kass and Martin Olsson for conversations which helped get this project started, and Vera Serganova and the MathOverflow community (especially Torsten Ekedahl, Jim Humphreys, Peter McNamara, David Speyer, and Angelo Vistoli) for their help with several technical points. We also thank Smiley for helping to track down many references. Finally, we would like to thank the anonymous referee for helpful suggestions and interesting questions.

\section{Local Construction of Toric Stacks}\label{sec:local-construction}

The main goal of this section is to prove Theorem \ref{thm:local-gms-data}.


\subsection{Colimits of Toric Monoids}

\begin{definition}\label{def:toric-monoid}
 A \emph{toric monoid} is any monoid of the form $\sigma\cap L$, where $\sigma$ is a pointed cone in a lattice $L$.
\end{definition}
\begin{remark}
 Toric monoids are precisely the finitely generated, commutative, torsion-free monoids $M$ so that $M\to M^\gp$ is injective and saturated.
\end{remark}
\begin{remark}
 Colimits exist in the category of toric monoids. A diagram of toric monoids $D$ induces a diagram of free abelian groups $D^\gp$. Let $L$ be the colimit of $D^\gp$ in the category of free abelian groups. Then the colimit of $D$ is the image in $L$ of the direct sum of all the objects of $D$. In particular, $\colim(D)^\gp = \colim(D^\gp)$.
\end{remark}

\begin{definition}\label{def:face-of-monoid}
 A \emph{face} of a monoid $M$ is a submonoid $F$ so that $a+b\in F$ implies $a,b\in F$.
\end{definition}
\begin{remark}
 For a toric monoid $\sigma\cap L$, the faces are precisely submonoids of the form $\tau\cap L$, where $\tau$ is a face of $\sigma$. So the faces of $\sigma\cap L$ are obtained as the vanishing loci of linear functionals on $L$ which are non-negative on $\sigma$.
\end{remark}
\begin{remark}\label{rmk:linear-functionals-from-faces}
 If $F$ is a face of a toric monoid $M$, then $F^\gp\to M^\gp$ is a saturated inclusion, so it is the inclusion of a direct summand. In particular, any linear functional on $F^\gp$ can be extended to a linear functional on $M^\gp$. Since $F$ is a face of $M$, there is a linear functional $\chi$ on $M^\gp$ which is non-negative on $M$ and vanishes precisely on $F$. Given any linear functional on $F^\gp$ which is non-negative on $F$, we extend it arbitrarily to a linear functional on $M^\gp$. By then adding a large multiple of $\chi$, we can guarantee that the extension is positive away from $F$.
\end{remark}

\begin{definition}\label{def:tight-diagram}
 Let $D$ be a finite diagram in the category of toric monoids (i.e.~$D$ is a collection of toric monoids $D_i$ and a collection of morphisms between the monoids). We say $D$ is \emph{tight} if
 \begin{enumerate}
 \item every morphism is an inclusion of a proper face,
 \item if $D_i$ appears in $D$, then all the faces of $D_i$ appear in $D$,
 \item the diagram commutes, and
 \item any two objects $D_i$ and $D_j$ in $D$ have a unique maximal common face in $D$.
\end{enumerate}
\end{definition}
\begin{remark}\label{rmk:tight-diagram}
 The motivation for Definition \ref{def:tight-diagram} is that for any fan $\Sigma$ on a lattice $L$, the diagram of toric monoids $\{\sigma\cap L|\sigma\in\Sigma\}$ is tight. The goal of this subsection is to show that any tight diagram of toric monoids is realized by a fan in this way. Indeed, Corollary \ref{cor:limits-of-toric-monoids} shows that any tight diagram of toric monoids can be realized by a subfan of the fan generated by a single cone.
\end{remark}

\begin{definition}\label{def:join-closed}
 A tight subdiagram $D^0$ of a tight diagram $D$ is \emph{join-closed} if whenever two objects of $D^0$ have a join in the poset $D$, they have the same join in the poset $D^0$. That is, for every pair of objects $D_i$ and $D_j$ of $D^0$, if they are both faces of an object $D_k$ of $D$, then the smallest face of $D_k$ containing $D_i$ and $D_j$ is in $D^0$.
\end{definition}

\begin{lemma}\label{lem:linear-functionals-from-join-closed}
 Let $D^0$ be a join-closed subdiagram of a tight diagram $D$. Suppose $\chi$ is a linear functional on $\colim(D^0)^\gp$. Then $\chi$ can be extended to a linear functional on $\colim(D)^\gp$. Moreover, if $\chi$ induces non-negative functions on all objects of $D^0$, then the extension can be chosen to be non-negative on all objects of $D$, and it can be chosen to be strictly positive away from $D^0$.
\end{lemma}
\begin{remark}
 By the universal property of a colimit, a linear functional on a colimit of groups is equivalent to a compatible collection of linear functionals on the groups in the diagram.
\end{remark}
\begin{proof}
 We induct on the size of $D\setminus D^0$. If it is empty, the result is clear. Otherwise, let $D_b$ be a maximal object of $D$ which is not in $D^0$. Let $D^1$ be the subdiagram of $D$ consisting of $D^0$ and all the faces of $D_b$ (including $D_b$ itself). Since $D_b$ is maximal, $D^1$ is a join-closed subdiagram of $D$. It suffices to extend the linear functional to $\colim(D^1)^\gp$.

 Since $D^0$ is join-closed, there is a maximum object $D_m$ of $D^0$ which is a face of $D_b$. We may extend $\chi|_{D_m}$ to a linear functional on $D_b$ as in Remark \ref{rmk:linear-functionals-from-faces}. If $\chi|_{D_m}$ is non-negative, we may choose the extension to be positive away from $D_m$.
\end{proof}

\begin{corollary}\label{cor:limits-of-toric-monoids}
 Let $D$ be a tight diagram of toric monoids with colimit $M$. Then for every object $D_i$ of $D$, $D_i\to M$ is an inclusion of a face.
\end{corollary}
\begin{proof}
 To show that $D_i\to M$ is an inclusion, it suffices to show that $D_i^\gp \to M^\gp$ is an inclusion, for which it suffices to show that the dual map is surjective. The subdiagram consisting of all the faces of $D_i$ is join-closed, so every linear functional on $D_i^\gp$ can be extended to a linear functional on $M^\gp$ by Lemma \ref{lem:linear-functionals-from-join-closed}, so the dual map is surjective.

 To show that $D_i$ is a face, it suffices to find a linear functional on $M^\gp$ which is non-negative on $M$ and vanishes exactly on $D_i$. Such a linear functional exists by Lemma \ref{lem:linear-functionals-from-join-closed}.
\end{proof}

\subsection{Constructing Toric Stacks Locally}

We saw in \cite[\S 5]{toricartin1} that every toric stack is a good moduli space of a canonical smooth toric stack. In this subsection, we show that we can construct a toric stack by starting with a smooth toric stack and specifying compatible good moduli space maps from an open cover. In other words, given a canonical stack morphism from a smooth toric stack, the property of being a toric stack can be checked locally. This result will be important in the proof of Theorem \ref{thm:main}.

\begin{theorem}\label{thm:local-gms-data}
 Let $\X$ be a stack with an action of a torus $T$ and a dense open $T$-orbit which is $T$-equivariantly isomorphic to $T$. Let $\Y\to \X$ be a morphism from a toric stack. Suppose $\X$ has a cover by $T$-invariant open substacks $\X_i$ which are toric stacks with torus $T$, and that the maps $\Y\times_\X \X_i\to \X_i$ are canonical stack morphisms (see \cite[Definition 5.1]{toricartin1}). Then $\X$ is a toric stack.
\end{theorem}
\begin{proof}
 Let $N=\hom_\gp(\GG_m,T)$ be the lattice of 1-parameter subgroups of $T$. Refining the cover, we may assume each $\X_i$ is of the form $\X_{\sigma_i,\beta_i\colon L_i\to N}$ with $\sigma_i$ a single cone. Moreover, we may assume that if $\X_{\sigma_i,\beta_i}$ is in the open cover, then the open substacks corresponding to the faces of $\sigma_i$ are as well. Then $\X$ is the colimit of this diagram of open immersions of toric stacks.

 The diagram of open immersions of the $\X_{\sigma_i,\beta_i}$ induces a diagram $D$ of toric monoids $\sigma_i\cap L_i$ (these monoids are well defined by \cite[Lemma B.16]{toricartin1}). We wish to show that $D$ is tight. By construction, the first three conditions of Definition \ref{def:tight-diagram} are satisfied. We need only to show that any two objects in $D$ have a unique maximal common face in $D$. For this, it suffices to show that the intersection of any two of the $\X_{\sigma_i,\beta_i}$ is cohomologically affine (and therefore corresponds to an element of the diagram). For each $\X_{\sigma_i,\beta_i}$, $\Y\times_\X\X_{\sigma_i,\beta_i}$ is a cohomologically affine toric open substack of $\Y$. Since $\Y$ is a toric stack, the intersection of two such substacks is cohomologically affine. By \cite[Remark 5.4]{toricartin1}, the intersection $(\Y\times_\X\X_{\sigma_i,\beta_i})\cap(\Y\times_\X\X_{\sigma_j,\beta_j})$ is the canonical
stack over $\X_{\sigma_i,\beta_i}\cap \X_{\sigma_j,\beta_j}$. In particular, $\X_{\sigma_i,\beta_i}\cap \X_{\sigma_j,\beta_j}$ has a toric surjection from a cohomologically affine toric stack, so it is cohomologically affine by \cite[Lemma B.7]{toricartin1}.

 The colimit of toric monoids $\sigma_i\cap L_i$ is of the form $\sigma\cap L$, where $L$ is the colimit of the $L_i$. By Corollary \ref{cor:limits-of-toric-monoids}, the $\sigma_i$ are faces of $\sigma$. By \cite[Proposition B.21]{toricartin1}, the induced morphisms $\X_{\sigma_i,\beta_i}\to \X_{\sigma,\beta}$ are the open immersions corresponding to the inclusions of the faces $\sigma_i\to \sigma$. The diagram of open immersions of the $\X_{\sigma_i,\beta_i}$ can therefore be realized as the diagram of inclusions of open substacks of $\X_{\sigma,\beta}$. Therefore, $\X$ is the union of these torus-invariant open substacks of $\X_{\sigma,\beta}$. In particular, it is toric.
\end{proof}

\begin{remark}\label{rmk:toric-stacks<coh-affine-toric-stacks}
 Note that the proof shows that $\X$ is an open substack of a \emph{cohomologically affine} toric stack. An interesting consequence is that any toric stack $\X_{\Sigma',\beta'\colon L'\to N}$ is an open substack of a cohomologically affine toric stack. Moreover, if $\Sigma'$ spans $L'$ (i.e.~$X_{\Sigma'}$ has no torus factors), then by applying the proof to the open cover of cohomologically affine torus-invariant open substacks, we see that $\X_{\Sigma',\beta'}$ is an open substack of a \emph{canonical} cohomologically affine toric stack $\X_{\sigma,\beta}$ (i.e.~one that depends only on $\X_{\Sigma',\beta'}$, not the stacky fan $(\Sigma',\beta')$). The corresponding stacky subfan of $(\sigma,\beta)$ is initial among all stacky fans which give rise to the toric stack $\X_{\Sigma',\beta'}$ (cf.~\cite[Appendix B]{toricartin1}).
\end{remark}

\section{Preliminary Technical Results}\label{sec:technical}

In this section, we gather technical results that will be used in the proofs of Theorems \ref{thm:local-structure}, \ref{thm:main-smooth}, and \ref{thm:main}.

\subsection{Some Facts About Stacks}
\begin{lemma}\label{lem:Weil=Cartier}
 Let $\Z$ be an irreducible Weil divisor (i.e.~a reduced irreducible closed substack of codimension $1$) of a stack $\X$. Suppose $U\to \X$ is a smooth cover. Then $\Z$ is a Cartier divisor of $\X$ if and only if $\Z\times_\X U$ is a Cartier divisor of $U$. In particular, on any smooth stack, every Weil divisor is Cartier.
\end{lemma}
\begin{proof}
 If $\I$ is the ideal sheaf of the Weil divisor $\Z$, then $\Z$ is Cartier if and only if $\I$ is a line bundle. One may verify that a quasi-coherent sheaf is locally free of a given rank locally in the smooth topology \cite[Theorem 11.4]{milne}. Since smooth morphisms are flat, the pullback to $U$ of ideal sheaf $\I$ is the ideal sheaf of the fiber product $\Z\times_\X U$.
\end{proof}

\begin{proposition}\label{prop:finite-index}
 Suppose $f\colon\X\to \Y$ is a quasi-compact representable \'etale morphism of quasi-separated algebraic stacks. Then $f$ induces finite-index inclusions on stabilizers of geometric points.
\end{proposition}
\begin{proof}
 Since $f$ is representable, it is faithful \cite[Proposition 2.4.1.3 with Corollary 8.1.2]{lmb}, so the induced maps on stabilizers are inclusions. Suppose that $x\colon \spec K\to \X$ is a geometric point, and let $G$ be the stabilizer of $f(x)$. The residual gerbe of $\Y$ at $f(x)$ must be trivial since $K$ is separably closed, so we have a stabilizer-preserving morphism \cite[Definition 2.1]{Alper:locquot} $BG\to \Y$ through which $f(x)$ factors. Since stabilizer-preserving morphisms are stable under base change, it suffices to show that the morphism $BG\times_\Y\X\to BG$ induces finite-index inclusions on stabilizers. Base changing along the $G$-torsor $\spec K\to BG$, we get a quasi-compact \'etale cover $U$ of $\spec K$, which must be a finite disjoint union of copies of $\spec K$:
 \[\xymatrix{
  U\ar[r]\ar[d]& BG\times_\Y\X\ar[r]\ar[d]^{\text{\'et, rep}} & \X\ar[d]^f\\
  \spec K\ar[r] & BG\ar[r] & \Y
 }\]
 We have that $U$ is a $G$-torsor over $BG\times_\Y\X$. If $H\subseteq G$ is the stabilizer of a point of $U$, then the orbit is isomorphic to $G/H$. Since $U$ is finite, any such $G/H$ must be finite, so $H$ must have finite index inside of $G$. The stabilizers at points of $BG\times_\Y\X$ are precisely such $H$.
\end{proof}

\begin{lemma}\label{lem:quot-has-affine-diagonal}
 Suppose $\X$ is an algebraic stack with affine diagonal. Suppose $G$ is an affine algebraic group with an action on $\X$. Then $[\X/G]$ has affine diagonal.
\end{lemma}

\begin{proof}
 The following diagram is cartesian:
 \[\xymatrix@C+1pc{
  \llap{$G\times \X\cong\;$} \X\times_{[\X/G]}\X \ar[r]^-\delta\ar[d] & \X\times \X \ar[d]\\
  [\X/G]\ar[r]^-{\Delta_{[\X/G]}} & [\X/G]\times [\X/G]
 }\]
 Since $\X\times \X\to [\X/G]\times[\X/G]$ is a smooth cover, it suffices to verify that the action morphism $\delta\colon G\times \X\to \X\times \X$ is affine.

 Composing $\delta$ with the projections $\X\times\X\to \X$ gives the projection and action maps $p_2,\alpha\colon G\times \X\to \X$. The projection $p_2$ is affine because $G$ is affine, and $\alpha$ is isomorphic to $p_2$, so it is also affine. We have that $\delta$ is then the composition $G\times \X\xrightarrow\Delta (G\times \X)\times (G\times \X)\xrightarrow{\alpha\times p_2} \X\times \X$. Since $\alpha\times p_2$ is a product of affine maps, it is affine. Since $G$ is affine, it has affine diagonal. By assumption, $\X$ also has affine diagonal, so $G\times \X$ has affine diagonal. So $\delta$ is a composition of affine morphisms.
\end{proof}

\begin{lemma}\label{lem:canonical-stack-affine-diagonal}
 If $\X$ has affine diagonal and $\Y\to \X$ is a canonical stack morphism (in the sense of \cite[Definition 5.6]{toricartin1}), then $\Y$ has affine diagonal.
\end{lemma}
\begin{proof}
 Consider the following diagram, in which the square is cartesian:
 \[\xymatrix{
  \Y\ar[r]_-{\Delta_{\Y/\X}}\ar@(ur,ul)[rr]^{\Delta_\Y} & \Y\times_\X\Y\ar[r]\ar[d] & \Y\times \Y\ar[d]\\
  & \X\ar[r]^-{\Delta_\X} & \X\times \X
 }\]
 Since $\Delta_\Y$ is a composition of $\Delta_{\Y/\X}$ and a pullback of $\Delta_\X$ (which is assumed to be affine), it suffices to show that $\Delta_{\Y/\X}$ is affine.

 Affineness can be verified locally on the base in the smooth topology, so we may assume $\Y=[X_{\ttilde \Sigma}/G_\Phi]$ and $\X=X_\Sigma$ (see \cite[Remark 5.2]{toricartin1}). In this case, $\Y$ has affine diagonal by Lemma \ref{lem:quot-has-affine-diagonal}, so in the above diagram $\Y$ and $\Y\times_\X\Y$ are both affine over $\Y\times \Y$, so $\Delta_{\Y/\X}$ is affine.
\end{proof}

\begin{lemma}
\label{l:affine-diag-gms}
If $\X\to\Y$ is a good moduli space morphism and $\X$ has affine diagonal, then $\Y$ has affine diagonal.
\end{lemma}
\begin{proof}
We must show that if $U_1$ and $U_2$ are affine schemes, and we have morphisms $U_i\to\Y$, then $U_1\times_\Y U_2$ is an affine scheme. Since $\X\times_{\Y}U_i\to U_i$ is a good moduli space morphism, $\X\times_{\Y}U_i$ is cohomologically affine. Note that
\[
\X\times_\Y(U_1\times_\Y U_2)\cong \X\times_{\Delta_\X,\X\times\X}(\X\times_{\Y}U_1)\times(\X\times_{\Y}U_2),
\]
so $\X\times_\Y(U_1\times_\Y U_2)$ is cohomologically affine, as $\Delta_\X$ is affine. Since $\X\times_\Y(U_1\times_\Y U_2)\to U_1\times_\Y U_2$ is a good moduli space morphism, \cite[Lemma 6.9(2)]{toricartin1} shows that $U_1\times_\Y U_2$ is affine.
\end{proof}

\begin{lemma}\label{lem:quot-has-red-stabilizers}
 Let $\X$ be an algebraic stack over a field $k$, with reductive stabilizers at geometric points, and let $G$ be a diagonalizable group over $k$ which acts on $\X$. Then $[\X/G]$ has reductive stabilizers at geometric points.
\end{lemma}
\begin{proof}
 Let $f\colon \spec K\to [\X/G]$ be a geometric point (i.e.~$K$ be a separably closed extension of the field $k$). Then $f$ is the image of some geometric point $\tilde f\colon \spec K\to \X$. We have the following diagram, in which the square is cartesian:
 \[\xymatrix{
   \spec K\ar[rd]_f\ar[r]^-{\tilde f} & \X\ar[r]\ar[d] & \spec k\ar[d]\\
  & [\X/G] \ar[r]^-\pi & BG
 }\]
 An automorphism $\phi$ of $f$ in $[X/G]$ induces an automorphism of $\pi\circ f$, which is a $K$-point of $G$. Since the square is cartesian, this image in $G$ is the identity if and only if $\phi$ is induced by an automorphism of $\tilde f$, so we get an exact sequence
 \[
  1\to \aut_\X(\tilde f)\to \aut_{[\X/G]}(f)\to G
 \]
 (exactness on the left follows from the fact that $\X\to [\X/G]$ is representable). So the stabilizer of the point of $[\X/G]$ is an extension of a subgroup of $G$ by the stabilizer of a pre-image in $\X$. Since $G$ is diagonalizable, any subgroup is diagonalizable, and so is reductive. An extension of reductive groups is reductive.
\end{proof}

Recall that a morphism $f\colon X\to Y$ is \emph{Stein} if $f_*\O_X=\O_Y$.

\begin{lemma}\label{lem:stein-open-immersion}
 Let $\X$ be a normal noetherian algebraic stack, and let $\U\subseteq \X$ be an open substack whose complement is of codimension at least $2$. Then the inclusion $\U\hookrightarrow \X$ is Stein.
\end{lemma}
\begin{proof}
 By cohomology and base change \cite[Proposition 9.3]{hartshorne}, the property of being Stein is local on the base in the smooth topology, so we may assume $\X=\spec R$, with $R$ a normal noetherian domain. Then $\U$ is a scheme, and we must show that any regular function on $\U$ arises as an element of $R$. Any regular function on $\U$ is a rational function on $R$, so it is of the form $f/g$, with $f,g\in R$. Since the complement of $\U$ is of codimension at least 2, we see that $f/g\in R_\mathfrak p$ for any codimension 1 prime $\mathfrak p$. A noetherian normal domain is the intersection in its fraction field of its localizations at codimension 1 primes \cite[Corollary 11.4]{eisenbud}, so $f/g\in R$.
\end{proof}

\begin{corollary}\label{cor:stein-birational-morphisms}
 Let $f\colon \X\to \Y$ be a morphism of normal noetherian algebraic stacks. Suppose there is an open substack $\U\subseteq \X$ so that $f|_\U$ is an isomorphism and so that $\U\subseteq \X$ and $f(\U)\subseteq \Y$ have complements of codimension at least $2$. Then $f$ is Stein.
\end{corollary}
\begin{proof}
 Let $i\colon \U\to \X$ be the inclusion. By Lemma \ref{lem:stein-open-immersion}, we have that $i$ and $f\circ i$ are Stein. It follows that $f_*\O_\X = f_*i_*\O_\U=\O_\Y$, so $f$ is Stein.
\end{proof}

\begin{proposition}\label{prop:smooth-maps-to-[An/Gm^n]}
 Let $D_1,\dots, D_n$ be effective Cartier divisors on a locally finite type scheme $X$ over a field $k$. Let $x\in X$ be a point at which $X$ is smooth and at which the $D_i$ have simple normal crossings. Then the induced morphism $\phi\colon X\to [\AA^n/\GG_m^n]$ is smooth at $x$.
\end{proposition}

\begin{remark}
 Smoothness of $X$ and the divisors at $x$ can be checked on a smooth cover of $X$, as can the property of having simple normal crossings. Therefore, this smoothness criterion applies to stacks as well.

 However, note that smoothness of the map $\phi\colon \X\to [\AA^n/\GG_m^n]$ \emph{does not} entail representability of the map. It simply means that for any smooth cover by a scheme $U\to\X$, the composite map $U\to [\AA^n/\GG_m^n]$ is smooth.
\end{remark}

\begin{proof}
 If some $D_i$ does not pass through $x$, then there is an open neighborhood of $x$ such that $\phi$ factors through $[(\AA^{n-1}\times \GG_m)/\GG_m^n]=[\AA^{n-1}/\GG_m^{n-1}]$. Smoothness can be checked on this neighborhood. We may therefore assume that all the divisors pass through $x$.

 We may verify formal smoothness at $x$ after restricting to the completed local ring $\hhat \O_{X,x}$. Since $X$ is locally of finite type, formal smoothness implies smoothness. By the Cohen structure theorem \cite[Theorem 7.7]{eisenbud}, $\hhat\O_{X,x}=k[[x_1,\dots, x_r]]$, and since the divisors have simple normal crossing, we may choose coordinates so that the divisor $D_i$ is the vanishing locus of the coordinate $x_i$. Then $\phi$ is a composition of three formally smooth morphisms: the ``inclusion'' of the complete local ring $\spec \hhat\O_{X,x}\to \AA^r$, the coordinate projection $\AA^r\to \AA^n$, and the quotient morphism $\AA^n\to [\AA^n/\GG_m^n]$.
\end{proof}

\subsection{The Weak Etale Slice Argument}
Here we prove a weak form of Luna's slice theorem. Our hypotheses are weaker than those in Luna's slice theorem (e.g.~we do not assume an action of a reductive group, only that the stabilizers are reductive), as is the conclusion (we do not get \emph{strong} \'etaleness). Since the hypotheses differ from the standard result significantly, we reproduce the proof here.

\begin{definition}\label{def:Ztimes^HG}
 Let $Z$ be a scheme with an action of a group scheme $H$, and let $H\subseteq G$ be a subgroup. Then $Z\times^H G$ (or $G\times^H Z$) denotes $(G\times Z)/H$, where the action of $H$ is given by $h\cdot (g,z)=(gh^{-1},h\cdot z)$.
\end{definition}
\begin{lemma}\label{lem:tangent-space-of-Gx^HZ}
 Let $Z$ be a scheme over a field $k$ of characteristic $0$. Let $G$ be a group scheme over $k$, and let $H\subseteq G$ be a subgroup. The tangent space to $G\times^H Z$ at the image of $(g,z)$ is $(T_gG\oplus T_zZ)/T_eH$, where the inclusion $T_eH\to T_gG\oplus T_zZ$ is induced by the inclusion $H\to G\times Z$, $h\mapsto (gh^{-1},h\cdot z)$.

 Moreover, $G\times^H Z$ is smooth at the image of a $k$-point $(g,z)$ if and only if $Z$ is smooth at $z$.
\end{lemma}
\begin{proof}
 We have a smooth map $G\times Z\to G\times^H Z$ whose fiber over the image of $(g,z)$ is $\{(gh^{-1},h\cdot z)|h\in H\}$. For any smooth map, the tangent space of an image point is the quotient of the tangent space of the point by the tangent space of the fiber at that point. This proves the first statement.

 We have that $G\times Z$ is an $H$-torsor over $G\times^H Z$ and a $G$-torsor over $Z$. Smoothness can be checked locally in the smooth topology.  $G$ and $H$ are smooth as we are over a field of characteristic zero, so we see that $Z$ is smooth at $z$ if and only if $G\times Z$ is smooth at $(g,z)$ if and only if $G\times^H Z$ is smooth at the image of $(g,z)$.
\end{proof}

\begin{lemma}\label{lem:qcompact-etale}
 Let $f\colon Y\to X$ be a quasi-compact morphism of schemes and $x\in X$ a point so that $f$ is \'etale at every point in the pre-image of $x$. Then there is an open neighborhood $U\subseteq X$ of $x$ so that the restriction $f^{-1}(U)\to U$ is \'etale.
\end{lemma}
\begin{proof}
 For every point $y\in f^{-1}(x)$, let $V_y\subseteq Y$ be an open neighborhood of $y$ so that $f|_{V_y}$ is \'etale. Since \'etale morphisms are open, $f(V_y)\subseteq X$ is open. Since the fiber $f^{-1}(x)$ is quasi-compact and \'etale over the point $x$, it is finite. Let $U=\bigcap_{y\in f^{-1}(x)}f(V_y)$. Then $U$ is an open neighborhood of $x$ such that $f^{-1}(U)\subseteq \bigcup_{y\in f^{-1}(x)}V_y$, so the induced morphism $f^{-1}(U)\to U$ is \'etale.
\end{proof}

\begin{proposition}[Weak \'etale slice argument]\label{P:luna}
 Let $G$ be an affine algebraic group acting on a quasi-affine scheme $X$ of finite type over an algebraically closed field $k$ of characteristic $0$. Suppose $x\in X$ is a $k$-point whose stabilizer $H\subseteq G$ is linearly reductive. Then there exists a connected locally closed $H$-invariant subscheme $Z\subseteq X$ such that $x\in Z$ and such that the induced morphism $Z\times^HG\to X$ is \'etale.
\end{proposition}
This roughly says that at a point with linearly reductive stabilizer $H$, a quotient stack $[X/G]$ is \'etale locally a quotient by $H$. Explicitly, we have the \'etale representable morphism $[Z/H]\cong [(Z\times^H G)/G]\to [X/G]$.
\begin{proof}
 We first consider the case where $X$ is smooth. Let $A=\O_X(X)$. By \cite[Lemma \href{http://math.columbia.edu/algebraic_geometry/stacks-git/locate.php?tag=01P9}{01P9}]{stacks-project}, the natural map $X\to \spec A$ is an open immersion, so we identify $X$ with an open subscheme of $\spec A$. Let $\m$ be the maximal ideal in $A$ corresponding to $x\in X$. The surjection $\m\to \m/\m^2\cong (T_xX)^*$ is $H$-equivariant. Since $H$ is linearly reductive, there is an $H$-equivariant splitting, which induces an $H$-equivariant ring homomorphism $\sym^*(\m/\m^2)\to A$ sending the positive degree ideal into $\m$. This corresponds to an $H$-equivariant map $\spec A\to T_xX$ sending $x$ to $0$ and inducing an isomorphism on tangent spaces at $x$. Since $X$ and $T_xX$ are smooth, the map is \'etale at $x$ \cite[\S 2.2, Corollary 10]{NeronModels}.

 The tangent space $T_xX$ has a natural action of $H$. The tangent space to the $G$-orbit through $x$ is an $H$-invariant subspace of $T_xX$. Since $H$ is linearly reductive, this subspace has an $H$-invariant complement $V$.
 \[\xymatrix@C+2pc{
   V \ari[d] & Z'\ar[l]_{\txt{$H$-eq}} \ari[d] \ari[r] & G\times^H Z'\ar[d]\\
   T_xX & X \ar@{=}[r] \ar[l]_{\txt{$H$-eq}} & X\\
 }\]
 We define $Z'$ as $V\times_{T_xX}X$. This is a closed $H$-invariant subscheme of $X$ which contains $x$. The map $Z'\to V$ is $H$-equivariant and is \'etale over $V$ at $x$. In particular, $Z'$ is smooth at $x$. The action of $G$ induces a morphism $G\times^H Z'\to X$. By Lemma \ref{lem:tangent-space-of-Gx^HZ}, $G\times^H Z'$ is smooth at the image of $(e,x)$, $X$ is smooth at $x$, and the map induces an isomorphism of tangent spaces since $T_xZ'\cong V$ is complementary to $T_x(G\cdot x)$. 
 By \cite[\S 2.2, Corollary 10]{NeronModels}, the map $G\times^HZ'\to X$ is \'etale at the image of $(e,x)$. Since the morphism is $G$-equivariant and every point in the fiber over $x$ is in a single $G$-orbit, it is \'etale at every point in the fiber, and therefore \'etale over a neighborhood of $x\in X$ by Lemma \ref{lem:qcompact-etale}. Since this map is $G$-equivariant, the locus in $X$ where it is \'etale is a $G$-invariant open neighborhood $U$ of $x$. Setting $Z=Z'\cap U$, we get that $Z\times^H G\to X$ is \'etale. This completes the proof in the case when $X$ is smooth.

 Now consider the case where $X$ is not smooth. We may choose a $G$-equivariant immersion of $X$ into a smooth scheme $X_0$. Indeed, $X_0$ can be chosen to be a finite-dimensional representation of $G$ \cite[Theorem 1.5]{Popov-Vinberg}. As shown above, there are representations $V\subseteq W$ of $H$, a $G$-invariant open neighborhood $U_0$ of $x$, and a closed subscheme $Z_0\subseteq U_0$ such that $Z_0=V\times_{W} U_0$ and $Z_0\times^H G\to U_0$ is \'etale. Setting $U=U_0\times_{X_0} X$ and $Z=Z_0\cap X$, we have the following cartesian diagram:
 \[\xymatrix{
  Z\times^H G\ar[r]\ar[d] & Z_0\times^H G\ar[d]\\
  U\ar[r]\ar[d] & U_0\ar[d]\\
  X\ar[r] & X_0
 }\]
 Since $Z_0\times^HG\to U_0$ is \'etale, so is $Z\times^H G\to U$.

 Finally, since $x$ is fixed by $H$, the connected component of $Z$ which contains $x$ is $H$-invariant. We may replace $Z$ by this connected component.
\end{proof}

\subsection{A Characterization of Pointed Toric Varieties}

The proof of the following proposition is due to Vera Serganova (see \cite{MO62182}).
\begin{proposition}\label{prop:unstable-closure-contains-highest-weight}
 Let $V$ be a representation of a linearly reductive group $G$ over a field $k$ of characteristic $0$, and let $Z=\overline{G\cdot v}\subseteq V$ be the closure of an unstable $G$-orbit $($i.e.~$0\in Z)$. If $Z$ is not contained in a direct sum of $1$-dimensional representations of $G$, then it contains a positive highest weight vector $($with respect to some Borel subgroup of $G)$.
\end{proposition}
\begin{proof}
 Note that $v$ itself cannot be in a direct sum of 1-dimensional representations. By the Hilbert-Mumford criterion \cite[Proposition 2.4]{git}, there is a 1-parameter subgroup $\gamma\colon \GG_m\to G$ so that $\gamma(t)\cdot v$ contains 0 in its closure. We have the weight space decomposition $V=\bigoplus_{i\in \ZZ} V_i$, where $V_i=\{x\in V|\gamma(t)x=t^ix\}$. Let $v=\sum_{i\ge p} v_i$, where $v_i\in V_i$ and $v_p\neq 0$. We may assume $p>0$ (replacing $\gamma$ by its inverse if necessary).

 Let $T$ be a maximal torus containing the image of $\gamma$, and let $B\subseteq H$ be a Borel subgroup containing $T$ so that $\gamma$ pairs non-negatively with all positive roots. Since only a finite number of weights appear in $V$, we may modify $\gamma$ so that it pairs \emph{positively} with all positive roots. If $v$ is a highest weight vector with respect to $B$, then we are done. Otherwise, there is some positive root $\alpha$ so that $e_\alpha\cdot v\neq 0$, with $e_\alpha\in \g_\alpha$, where $\g_\alpha$ is the root space corresponding to $\alpha$ in the Lie algebra of $G$. Let $\exp(te_\alpha)\cdot v = \sum_{i\ge p} f_i(t)$, where $f_i(t)\in V_i\otimes k[t]$. Let $m_i=\deg f_i$. Since $\alpha$ pairs positively with $\gamma$ (and $e_\alpha\cdot V_i\subseteq V_{i+\<\gamma,\alpha\>}$), we have $e_\alpha\cdot v\in \bigoplus _{i>p}V_i$, so $m_p=0$. Moreover, since $e_\alpha\cdot v\neq 0$, some $m_i$ is positive.

 \begin{window}[0,r,%
     {\begin{tikzpicture}[scale=.8]
       \draw[<->] (4,0) node[anchor=north east] {$i$} -- (0,0) -- (0,2.5) node[anchor=north east] {$m_i$};
       \draw plot[only marks,mark=*] coordinates{(1.5,0) (2,1) (2.5,1) (3,1.5) (3.5,.5)}
       (1.5,0) node[anchor=north] {$p$};
       \draw (0,0) -- (4,2) node[anchor=south east] {slope$\;=a/b$};
      \end{tikzpicture}%
     },]
 Let $\frac ab\in \QQ$ be the rational number so that $m_i \le \frac ab i$ for all $i$ and $m_j=\frac ab j$ for some $j$. Consider the function $g\colon \AA^1\to V$ given by $g(t) = \sum t^{a\cdot i}f_i(t^{-b})$. Note that this is well defined since $\deg f_i = m_i \le \frac ab i$ for all $i$, so $\deg\bigl(t^{a\cdot i}f_i(t^{-b})\bigr)=a\cdot i - b\cdot m_i\ge 0$. Note also that $g(0)\neq 0$ since $m_j=\frac ab j$ for some $j$. For $t\neq 0$, we have that $g(t) = \gamma(t^a)\cdot \exp(t^{-b}e_\alpha)\cdot v \in Z$. Since $Z$ is closed, we have that $g(0)\in Z$. Note that the minimal weight (with respect to $\gamma$) appearing in $g(0)$ is greater than $p$ and that $g(0)$ does not lie in a direct sum of 1-dimensional representations since it is in the image of $e_\alpha$. Since $V$ is finite-dimensional, repeating this procedure a finite number of times produces a positive highest weight vector in $Z$.\qedhere
 \end{window}
\end{proof}

\begin{proposition}\label{prop:characterization-pointed-toric}
 Suppose $Z$ is an irreducible affine scheme of finite type over an algebraically closed field $k$ of characteristic $0$, with an action of a linearly reductive group $H$. Suppose that $x\in Z$ is an $H$-invariant $k$-point, that $Z$ contains a dense open stabilizer-free orbit, and that the stabilizer of each $k$-point of $Z$ is linearly reductive. Then $H$ is a torus. In particular, if $Z$ is reduced and normal, it is a toric variety.
\end{proposition}
\begin{proof}
 Since $H$ is isomorphic to a dense open subscheme of $Z$, it is irreducible. Let $Z=\spec A$, and let $\m\subseteq A$ be the maximal ideal corresponding to $x$. We may choose a finite-dimensional $H$-invariant subspace $V^*\subseteq \m$ such that $V^*$ generates $A$ as a $k$-algebra. Then $\spec A\to \spec(\sym^*(V^*))=V$ is a closed $H$-equivariant immersion of $Z$ into a finite-dimensional representation of $H$, sending $x$ to the origin. Since $Z$ contains a dense open stabilizer-free $H$-orbit, the subrepresentation spanned by $Z$ is faithful. If $Z$ is contained in a direct sum of 1-dimensional representations, then $H$ is diagonalizable, so it is a torus. Otherwise, $Z$ contains a positive highest weight vector $v$ by Proposition \ref{prop:unstable-closure-contains-highest-weight}.
 Then $v$ is stabilized by the unipotent radical of some Borel subgroup of $H$. Since $v$ has reductive stabilizer, it must also be stabilized by the unipotent radical of the opposite Borel subgroup, and so by the derived group of $H$, contradicting the assumption that it is a \emph{positive} weight vector.
\end{proof}

\section{The Local Structure Theorem}
\label{sec:local-structure}

The main result of this section is Theorem \ref{thm:local-structure}. Together with Lemma \ref{lem:check-toric-after-quotient}, this theorem serves as our main tool for showing that a stack is toric.

\begin{lemma}\label{lem:check-toric-after-quotient}
 Let $\X$ be an algebraic stack over a field $k$ with an action of a torus $T$ and a dense open substack which is $T$-equivariantly isomorphic to $T$. Then $\X$ is a toric stack if and only if $[\X/T]$ is a toric stack.
\end{lemma}
\begin{proof}
 If $\X=[X/G]$ is a toric stack, where $X$ is a toric variety and $G\subseteq T_X$ is a subgroup of the torus, then $T=T_X/G$. We see that $[\X/T]\cong [X/T_X]$ is a toric stack.

 Now suppose $[\X/T]=[X/G]$ is a toric stack, where $X$ is a toric variety and $G\subseteq T_X$ is a subgroup of the torus. Since $[\X/T]$ has a dense open point, we have that $G=T_X$ is the torus of $X$. Consider the following cartesian diagram:
 \[\xymatrix@!0 @C+4pc @R+1pc{
  *+<10pt>{T_X\times T} \ari[dr]\ar[dd]\ar[rr] && T\ari[dr]\ar[dd]|{\hole}\\
  & X\times_{[\X/T]}\X \ar[rr]\ar[dd] && \:\X\:\ar[dd]^{\txt{$T$-torsor}}\\
  *+<13pt>{T_X}\ari[dr]\ar[rr]|{\hole} && *+<13pt>{[T_X/T_X]=[T/T]=\ast\quad} \ari[dr]\\
  & X\ar[rr] && [X/T_X]\cong [\X/T]\phantom{\;\cong [X/T_X]}
 }\]
 The stack $X\times_{[\X/T]}\X$ has an action of the torus $T_X\times T$ and a dense open substack isomorphic to $T_X\times T$. Since $X\times_{[\X/T]}\X$ is a $T$-torsor over $X$, it is a normal separated scheme, so it is a toric variety with torus $T_X\times T$. It is also a $T_X$-torsor over $\X$, so $\X=[(X\times_{[\X/T]}\X)/T_X]$ is a toric stack.
\end{proof}

\begin{definition}\label{def:global-type}
  A finite type stack $\X$ over a field $k$ is \emph{of global type} if every geometric point is in the image of an \'etale representable map $[U/G]\to \X$, where $U$ is a quasi-affine $k$-scheme and $G$ is an affine algebraic group.
\end{definition}

The following lemma shows that Definition \ref{def:global-type} agrees with \cite[Definition 2.1]{rydh:noetherian}.
\begin{lemma}
\label{l:global-type-definitions-agree}
A finite type stack $\X$ over a field $k$ is of global type if and only if there is a finitely presented \'etale representable surjection $[U/GL_n]\to\X$, where $U$ is a quasi-affine $k$-scheme.
\end{lemma}
\begin{proof}
  It is clear that every stack satisfying this condition is of global type.

  For the converse, we first show that $G$ can always be taken to be $GL_n$. By \cite[Lemma 3.1]{Totaro}, every affine algebraic group $G$ over a field has a faithful representation $G\hookrightarrow GL_n$ so that $GL_n/G$ is quasi-affine. The projection $U\times GL_n\to GL_n$ is quasi-affine, so the induced quotient map $U\times^G GL_n\to GL_n/G$ is quasi-affine (quasi-affineness of a morphism can be checked fppf locally), so $U\times^G GL_n$ is quasi-affine. We now have that $[U/G]\cong [(U\times^G GL_n)/GL_n]$ is a quotient of a quasi-affine scheme by $GL_n$.

  Next, since $\X$ is of finite type over a field, it is quasi-compact, so it has an \'etale surjection from a finite number of $[U_i/GL_{n_i}]$, where the $U_i$ are quasi-affine. We show that the $n_i$ may be increased until they are all equal. If $U$ is quasi-affine with affine envelope $W=\spec \O_U(U)$, an action of $GL_n$ on $U$ induces an action of $GL_n$ on $W$. The $GL_n$-equivariant immersion $U\to W$ induces a $GL_{n+1}$-equivariant immersion $U\times^{GL_n}GL_{n+1}\to W\times^{GL_n}GL_{n+1}$, but since $W\times^{GL_n}GL_{n+1} = (W\times GL_{n+1})/GL_n$ is a quotient of an affine scheme by a free action of a linearly reductive group, it is affine, so $U\times^{GL_n}GL_{n+1}$ is quasi-affine, and $[U/GL_n]=[(U\times^{GL_n}GL_{n+1})/GL_{n+1}]$.

  Thus any stack $\X$ which is of global type has a representable \'etale surjection $[V /GL_n]\to \X$, where $V=\bigsqcup_i U_i$ is a finite union of quasi-affine schemes. So $V$ is quasi-affine, and therefore quasi-compact (recall that quasi-compactness is part of the definition of quasi-affineness). Since $\X$ is of finite type and $V$ is smooth over $\X$, $V$ is of finite type, so the morphism $[V /GL_n]\to \X$ is finitely presented.
\end{proof}

\begin{remark}\label{rmk:everything-is-global-type}
  We know of no stack of finite type over a field which has affine diagonal and is \emph{not} of global type. Here we summarize what is known:
  \begin{enumerate}\setcounter{enumi}{-1}
    \item \label{alper-hall-rydh} As this paper went to press, we learned of a forthcoming result of Alper, Hall, and Rydh which implies that for a finite type stack $\X$ with affine diagonal over an algebraically closed field, any closed point with linearly reductive stabilizer is in the image of a representable etale morphism $[U/G]\to \X$, where $U$ is affine and $G$ is the stabilizer of the point. Since the images of such morphisms are open, this shows that all global type hypotheses in this paper are superfluous.
    \item \label{everything-gt-1} If $\X$ is covered by \'etale representable morphisms from stacks of the form $[X/G]$, where $X$ is a \emph{normal noetherian scheme} over $k$ and $G$ is an affine algebraic group, then $\X$ is of global type. To see this, let $G^\circ\subseteq G$ be the connected component of the identity, and let $H=G/G^\circ$.
    Then $[X/G^\circ]$ is of global type by \cite[Remark 2.3]{rydh:noetherian}, and $[X/G^\circ]\to [[X/G^\circ]/H]=[X/G]$ is \'etale and representable since it is a torsor under the discrete finite group $H$, so $[X/G]$ is of global type.
    \item Totaro has shown \cite[Theorem 1.1]{Totaro} that a normal noetherian stack $\X$ is of the form $[U/GL_n]$ with $U$ quasi-affine if and only if $\X$ has the resolution property. Thus, a stack is of global type if and only if every point has a representable \'etale neighborhood which has the resolution property.
    \item Suppose $\X=[X/G]$ with $X$ a quasi-affine scheme and $G$ an affine algebraic group, and suppose $\X$ has an action of an algebraic group $T$. We do not know if $[\X/T]$ must be of global type. In practice, the action of $T$ on $\X$ is usually induced by an action of $\ttilde G$ on $X$, where $\ttilde G$ is an extension of $T$ by $G$. Then $\ttilde G$ is a $G$-torsor over $T$, so $\ttilde G$ is affine. Hence, $[\X/T] = [X/\ttilde G]$ is of global type by (\ref{everything-gt-1}) above.
    \item In \cite[Questions 2.12]{rydh:noetherian}, Rydh asks if every quasi-compact stack with quasi-affine diagonal is of global type.
  \end{enumerate}
\end{remark}

\begin{theorem} \label{thm:local-structure}
 Suppose $\X$ is an Artin stack of finite type over an algebraically closed field $k$ of characteristic $0$, with a dense open $($non-stacky$)$ $k$-point. Let $\xi\colon \spec k\to \X$ be a point. Suppose
 \begin{enumerate}
  \item $\X$ is normal,
  \item $\X$ has affine diagonal,
  \item $\X$ has linearly reductive stabilizers at geometric points, and
  \item $\X$ is of global type $($see Remark \ref{rmk:everything-is-global-type}$)$.
 \end{enumerate}
 Then there is an affine toric variety $X$ with torus $T$ and an open immersion $[X/T]\hookrightarrow \X$ sending the distinguished closed point of $[X/T]$ to $\xi$.\footnote{Note that $[X/T]$ has a distinguished closed point, even if $X$ does not. An affine toric variety $X$ can only fail to have a distinguished closed point if it is of the form $X'\times T_0$, where $X'$ has a distinguished closed point and $T_0$ is a torus. In this case, $[X/T]\cong [X'/(T/T_0)]$.}
\end{theorem}

\begin{proof}
 Let $[U/G]\to \X$ be a representable \'etale morphism with $\xi$ in its image, where $U$ is a quasi-affine scheme and $G$ is an affine algebraic group. Let $x\in U$ be a $k$-point mapping to $\xi$, and let $H\subseteq G$ be the stabilizer of $x$. Since the morphism $[U/G]\to \X$ is \'etale representable, Proposition \ref{prop:finite-index} implies that the stabilizers of geometric points of $[U/G]$ are finite index subgroups of the stabilizers of geometric points of $\X$. That is, given a point $y\colon \spec k\to [U/G]$, $\stab_\X(y)/\stab_{[U/G]}(y)$ is finite, and therefore affine. Since the stabilizers of geometric points of $\X$ are linearly reductive, Matsushima's criterion (see \cite[Proposition 12.15(i)]{Alper:good}) implies that the stabilizers of geometric points of $[U/G]$ are linearly reductive.

 Applying Proposition \ref{P:luna}, there is a connected locally closed $H$-invariant subscheme $Z\subseteq U$ so that $x\in Z$ and $Z\times^H G\to U$ is \'etale. The morphism $Z\to \X$ is smooth, as it is the composition of smooth morphisms $Z\to [Z/H]\cong [(Z\times^H G)/G]\to [U/G]\to \X$. Since $\X$ is normal, and normality is local in the smooth topology, $Z$ is normal. Since $Z$ is also connected, it is irreducible.

 The map $[Z/H]\to [U/G]\to \X$ is \'etale and representable. Base changing $[Z/H]\to \X$ to the dense open $k$-point of $\X$, we get an irreducible \'etale cover of $\spec k$, which must be trivial since $k$ is algebraically closed. In particular, $[Z/H]$ has a dense open $k$-point, so $Z$ contains a dense open stabilizer-free $H$-orbit.

 Next we show that $Z$ must be affine. Let $A=\O_Z(Z)$. Since $Z$ is quasi-affine, it is a dense open subscheme of $\spec A$ \cite[Lemma \href{http://math.columbia.edu/algebraic_geometry/stacks-git/locate.php?tag=01P9}{01P9}]{stacks-project}. The action of $H$ on $Z$ induces an action of $H$ on $\spec A$. By \cite[Theorem 1.1]{git}, $(\spec A)/H = \spec (A^H)$ is a good quotient. Since $\spec A$ contains a dense open copy of $H$, any $H$-invariant regular function must be constant, so the good quotient is $\spec k$. It follows 
 that the closures of any two $H$-orbits intersect. But, $x\in Z$ is a closed $H$-orbit and $Z\subseteq \spec A$ is an $H$-invariant open neighborhood of $x$, so $Z=\spec A$.

 By the same argument we used in the first paragraph of this proof, the stabilizers of $[Z/H]$ are linearly reductive. Since $Z$ is smooth over $\X$, it is normal and reduced. By Proposition \ref{prop:characterization-pointed-toric}, $H$ is a torus and $Z$ is a toric variety.

 Finally, we have an \'etale representable map $[Z/H]\to \X$ whose image is an open substack. Replacing $\X$ by this open substack, we may assume the map is surjective. Now we have that $Z\to \X$ is a smooth cover. Consider the following cartesian diagram:
 \[\xymatrix{
  Z\times_\X Z\ar[d] \ar[r] & Z\ar[d]^{H\text{-torsor}}\\
  Y\ar[r]\ar[d] & [Z/H]\ar[d]\\
  Z\ar[r] & \X
 }\]
 Since $Z$ is affine and $\X$ has affine diagonal, we have that $Z\times_\X Z$ is affine. This affine space is the total space of an $H$-torsor over $Y$. Since $[Z/H]\to \X$ is representable, $Y$ is an algebraic space. Since $H$ is linearly reductive, $Y$ is an affine scheme \cite[Theorem 1.1]{git}. Since $Y$ and $Z$ are both affine, $Y\to Z$ is separated. Separatedness is local on the base in the smooth topology, so $[Z/H]\to \X$ is separated.

 Now $[Z/H]\to \X$ is representable, separated, \'etale, birational, and surjective, so it is an isomorphism by Zariski's Main Theorem \cite[Theorem 16.5]{lmb}.
\end{proof}

\section{Main Theorem: Smooth Case}\label{sec:main-smooth}
In this section, we show that every smooth ``abstract toric stack'' is actually a toric stack.  That is, we show that with suitable hypotheses, a smooth Artin stack $\X$ with dense open torus $T$ whose action on itself extends to $\X$ comes from a stacky fan.
\begin{lemma}\label{lem:smooth=>codim-preserving}
 Suppose $f\colon \X\to \Y$ is a smooth (but not necessarily representable) morphism of Artin stacks. Then $f$ is codimension-preserving: if $\Z\subseteq \Y$ is a closed substack of codimension $d$, then $\Z\times_\Y\X\subseteq \X$ is of codimension $d$.
\end{lemma}
\begin{proof}
 Let $\pi\colon U\to \X$ be a smooth cover by a scheme. Then $g=f\circ\pi\colon U\to \Y$ is smooth and representable, so it is open and codimension-preserving. If $\Z\subseteq \Y$ is a closed substack of codimension $d$, then $\Z\times_\Y U\subseteq U$ is closed of codimension $d$. On the other hand, $U\to \X$ is codimension-preserving, and $\Z\times_\Y U = (\Z\times_\Y \X)\times_\X U$, so $\Z\times_\Y \X\subseteq \X$ is of codimension $d$.
\end{proof}

\begin{theorem}\label{thm:main-smooth}
 Let $\X$ be a smooth Artin stack of finite type over an algebraically closed field $k$ of characteristic $0$. Suppose $\X$ has an action of a torus $T$ and a dense open substack which is $T$-equivariantly isomorphic to $T$. Then $\X$ is a toric stack if and only if the following conditions hold:
 \begin{enumerate}
  \item $\X$ has affine diagonal,
  \item geometric points of $\X$ have linearly reductive stabilizers, and
  \item $[\X/T]$ is of global type (see Remark \ref{rmk:everything-is-global-type}).
 \end{enumerate}
\end{theorem}

\begin{proof}
 It is clear that smooth toric stacks satisfy the conditions, so we focus on the converse.

 By Lemma \ref{lem:check-toric-after-quotient}, it suffices to check that $[\X/T]$ is a toric stack. By Lemma \ref{lem:quot-has-affine-diagonal}, $[\X/T]$ has affine diagonal. By Lemma \ref{lem:quot-has-red-stabilizers}, $[\X/T]$ has linearly reductive stabilizers. Thus, we have reduced to the case where $T$ is trivial and $\X$ has a dense open $k$-point.

 Consider the (finite) set of irreducible divisors of $\X$. By Lemma \ref{lem:Weil=Cartier}, these divisors are Cartier, so they are induced by line bundles $\L_1,\dots, \L_n$ with non-zero global sections $s_i\in \Ga(\X,\L_i)$. These line bundles and sections induce a morphism $\X\to [\AA^n/\GG_m^n]$. We will show that this morphism is an open immersion---and therefore that $\X$ is a toric stack---by induction on $n$.

 \subsubsection*{The case $n=0$}
 If $\X$ has no divisors, then we claim that $\X=\spec k$. By Theorem \ref{thm:local-structure}, every point of $\X$ has an open neighborhood of the form $[X/T_X]$, where $X$ is a toric variety and $T_X$ is its torus. Every point of a toric variety lies either in the torus or on a torus-invariant divisor. Since $\X$ has no divisors, $X$ can have no torus-invariant divisors.  It follows that $X$ must be a torus, and so $\X$ is covered by its dense open point.

 \subsubsection*{The case $n=1$}

 Suppose $\D\subseteq \X$ is the unique divisor. Our aim is to show that the morphism $\X\to [\AA^1/\GG_m]$ is an isomorphism.

 Applying Theorem \ref{thm:local-structure} to points of $\D$, we see that $\D$ has a dense (stacky) geometric point and that any other point must lie on the intersection of two or more distinct divisors (because this is true for torus-invariant divisors on a smooth toric variety). Since $\D$ is the unique divisor of $\X$, it has only one geometric point $p$. Aside from this point, $\X$ has only one other point: the dense open point. Applying Theorem \ref{thm:local-structure} around $p$, we get an open neighborhood of the form $[X/T]$, where $X$ is a toric variety and $T$ is its torus. But any open neighborhood $p$ must be all of $\X$, so $\X=[X/T]$ is a toric stack. Moreover, the toric variety $X$ has precisely one torus-invariant divisor, so $[X/T]=[\AA^1/\GG_m]$.

 \subsubsection*{The general case $n\ge 2$}

 Suppose $\D_1,\dots, \D_n$ are the divisors cut out by the sections $s_i\in \Ga(\X,\L_i)$. By induction on $n$, $\X\setminus \D_i$ is a smooth toric stack, so the morphism $\X\setminus \D_i\to [\AA^{n-1}/\GG_m^{n-1}]$ is an open immersion. On the other hand, this morphism is part of the following cartesian diagram:
 \[\xymatrix@!0 @R+1.5pc @C+4.5pc{
  & \X\setminus \D_i \ari[r]\ar[d] & \X\ar[d]\\
  \llap{$[\GG_m/\GG_m]\times\;$} [\AA^{n-1}/\GG_m^{n-1}]\ar@{=}[r] & [\AA^{n-1}/\GG_m^{n-1}] \ari[r] & [\AA^n/\GG_m^n]
 }\]
Therefore, we see that the morphism $\X\to [\AA^n/\GG_m^n]$ restricts to an open immersion $\U:=\X\setminus (\D_1\cap\cdots\cap \D_n) \to [(\AA^n\setminus \{0\})/\GG_m^n]$. If $\D_1\cap \cdots\cap \D_n=\varnothing$, then we are done, so we may assume $\Z=\D_1\cap \cdots\cap \D_n$ is non-empty. Then any subset of divisors intersect, but the divisors are distinct, so $\X\to [\AA^n/\GG_m^n]$ is set-theoretically surjective. In particular, $\U\to [(\AA^n\setminus \{0\})/\GG_m^n]$ is an isomorphism. By Theorem \ref{thm:local-structure}, $\Z$ is of codimension $n\ge 2$. So by Lemma \ref{lem:stein-open-immersion}, $\X\to [\AA^n/\GG_m^n]$ is Stein.

 By Theorem \ref{thm:local-structure}, $\X$ is Zariski locally a quotient of a smooth toric variety. In particular, the divisors are smooth and have simple normal crossings, so by Proposition \ref{prop:smooth-maps-to-[An/Gm^n]}, $\X\to [\AA^n/\GG_m^n]$ is smooth (but may not be representable). So $\X\times_{[\AA^n/\GG_m^n]}\X\to [\AA^n/\GG_m^n]$ is smooth and is an isomorphism over the complement of the closed point of $[\AA^n/\GG_m^n]$. Since smooth maps are codimension preserving (Lemma \ref{lem:smooth=>codim-preserving}), the complement of $\U\cong\U\times_{[\AA^n/\GG_m^n]}\U\subseteq \X\times_{[\AA^n/\GG_m^n]}\X$ is of codimension $n\ge 2$. In particular, the diagonal $\Delta_{\X/[\AA^n/\GG_m^n]}$ is Stein by Lemma \ref{lem:stein-open-immersion}.

 Consider the following diagram, in which the square is cartesian:
 \[\xymatrix@C+2pc{
  \X\ar[r]_-{\Delta_{\X/[\AA^n/\GG_m^n]}} \ar@(ur,ul)[rr]^{\Delta_\X} & \X\times_{[\AA^n/\GG_m^n]}\X\ar[r]\ar[d] & \X\times \X\ar[d]\\
  & [\AA^n/\GG_m^n]\ar[r]^-{\Delta_{[\AA^n/\GG_m^n]}} & [\AA^n/\GG_m^n]\times [\AA^n/\GG_m^n]
 }\]
 Since $\Delta_\X$ and $\Delta_{[\AA^n/\GG_m^n]}$ are affine, we see that $\Delta_{\X/[\AA^n/\GG_m^n]}$ is affine.

 Now $\Delta_{\X/[\AA^n/\GG_m^n]}$ is Stein and affine, so it is an isomorphism. Thus, $\X\to [\AA^n/\GG_m^n]$ is a monomorphism, so it is representable \cite[Corollary 8.1.2]{lmb}, separated, and quasi-finite. Since $[\AA^n/\GG_m^n]$ is normal, Zariski's Main Theorem \cite[Theorem 16.5]{lmb} implies that $\X\to [\AA^n/\GG_m^n]$ is an open immersion.
\end{proof}

\subsection{Counterexamples}
This subsection gives interesting examples of stacks which look like they might be toric stacks, but are not.  For each example, we show how the conditions of Theorem \ref{thm:main-smooth} fail.

To begin, there are varieties $X$ that contain a dense open torus $T$, on which $T$ cannot possibly act. For example, blowing up a torus-non-invariant point on a divisor of a toric variety will produce such a variety. When working with algebraic spaces and stacks, the action can fail to extend for more subtle reasons.
\begin{example}[Torus action does not always extend]\label{eg:bug-eyed}
 Let $U$ be the affine line with a doubled origin over a field of characteristic not equal to $2$. Let $\ZZ/2$ act on $U$ by $x\mapsto -x$ and switching the two origins. Then $X=[U/(\ZZ/2)]$ is a smooth algebraic space with a dense open torus $[\GG_m/(\ZZ/2)]\cong \GG_m$. This space is a ``bug-eyed cover'' of $\AA^1$ \cite{kollar:bug-eyed}. We claim that the torus cannot act on $X$.

 If it did, the \'etale cover $\AA^1\to X$ would be toric, inducing the degree 2 map of tori $\GG_m\to \GG_m/(\ZZ/2)$. This map induces an isomorphism of $\GG_m$-representations between the tangent space to $\AA^1$ at $0$ and the tangent space to $X$ at ``the bug eye''. It would follow that the 1-dimensional weight 1 representation of $\GG_m$ (i.e.~the tangent space to $\AA^1$ at 0) factors through the degree 2 map $\GG_m\to \GG_m/(\ZZ/2)$, which it clearly does not.
\end{example}

Although we have shown that the previous example is not a toric stack, it is nonetheless interesting to observe that it can be extended.
\begin{example}
 Consider the stack $\X=[\AA^2/(\ZZ/2\ltimes \GG_m)]$, with the action given by $(0,t)\cdot (x,y)=(tx,t^{-1}y)$ and $(1,1)\cdot(x,y)=(-y,x)$. This contains the ``bug-eyed cover'' from the previous example as an open substack (it is the image of $\AA^2\setminus \{0\}$).

 What makes $\X$ particularly interesting is that it is a smooth stack with a dense open torus (whose action does not extend, of course) so that the complement of the torus is a single \emph{singular} divisor (the union of the axes in $\AA^2$ is smooth over this divisor).
\end{example}

\begin{remark}\label{rmk:affine-diagonal<->separated}
 A notable difference between toric stacks and toric varieties is that toric varieties are required to be separated. Artin stacks are almost never separated, but the affine diagonal condition seems to play the role of separatedness. Heuristically, toric stacks are entirely controlled by their torus-invariant divisors (this is made precise by \cite[Theorem 7.7]{toricartin1} and the canonical stack construction in \cite[Section 5]{toricartin1}). The condition that a stack have affine diagonal ``forces all non-separatedness to occur in codimension 1'' and therefore be controlled by the combinatorics.
\end{remark}

\begin{example}[Non-affine diagonal]\label{Eg:non-affine-diag}
 The affine plane with a doubled origin is a scheme with a torus action satisfying nearly all the conditions of Theorem \ref{thm:main-smooth}, except it does not have affine diagonal.

 Note however, that the affine \emph{line} with a double origin does have affine diagonal, and is in fact a toric stack. It is $[(\AA^2\setminus \{0\})/\GG_m]$, where $\GG_m$ acts by $t\cdot (x,y)=(tx,t^{-1}y)$.
\end{example}

In the world of stacks, non-affine diagonals can occur in stranger ways as well.
\begin{example}[Non-\emph{separated} diagonal]\label{Eg:B-of-non-sep}
 Let $G$ be the affine line with a doubled origin, regarded as a relative group over $\AA^1$. The fibers away from the origin are trivial, and the fiber over the origin is given the structure of $\ZZ/2$. We see that $G\to \AA^1$ is an \'etale relative group scheme. 

 Let $\X=[\AA^1/G]$, where $G$ acts trivially on $\AA^1$. Since $\X$ has a representable \'etale cover by $\AA^1$, it is of finite type, normal, and of global type. Moreover, it has linearly reductive stabilizers at geometric points. It contains a dense open torus $T\cong \GG_m$ which acts on it. However, $\X$ has \emph{non-separated} diagonal, so it is not a toric stack.
\end{example}

In Theorem \ref{thm:main-smooth}, the condition that $\X$ have linearly reductive stabilizers is necessary. It is easy to produce many examples of stacks that satisfy all the other conditions of the theorem, but fail to be toric stacks.
\begin{example}[Non-reductive stabilizers]\label{Eg:non-diag-stabilizers}
 If $X$ is any smooth scheme of finite type with an action of an affine group $G$ and a dense open copy of $G$, then $\X=[X/G]$ has a dense open torus (the \emph{trivial} torus $[G/G]$) which acts (trivially). Since $G$ is affine, $\X$ has affine diagonal by Lemma \ref{lem:quot-has-affine-diagonal}. By Remark \ref{rmk:everything-is-global-type}(\ref{everything-gt-1}), $\X$ is of global type.

 As a concrete case, consider the stack $\X=[M_{2\times 2}/GL_2]$, where the action of $GL_2$ on $M_{2\times 2}\cong \AA^4$ is given by left multiplication. This stack does not satisfy the conditions of Theorem \ref{thm:main-smooth} since not all stabilizers are linearly reductive. For example, $\stab\smat{1&0\\ 0&0}\cong \GG_m\ltimes \GG_a$.
\end{example}

\section{Main Theorem: Non-smooth Case}\label{sec:main}
We now extend the results of the previous section to handle the case of singular stacks.  In the introduction of \cite{toricartin1}, we mentioned that smooth toric stacks can serve as better-behaved substitutes for toric varieties: sometimes it is easier to prove a result on the canonical smooth stack overlying a toric variety (see \cite[\S5]{toricartin1}) and then ``push the result down'' to the toric variety.  This section yields a concrete instance of this philosophy.  Indeed, we will prove our main theorem for singular toric stacks $\X$ by making use of the canonical stack over $\X$ and Theorem \ref{thm:main-smooth}.
\begin{theorem}\label{thm:main}
 Let $\X$ be an Artin stack of finite type over an algebraically closed field $k$ of characteristic $0$. Suppose $\X$ has an action of a torus $T$ and a dense open substack which is $T$-equivariantly isomorphic to $T$. Then $\X$ is a toric stack if and only if the following conditions hold:
 \begin{enumerate}
  \item $\X$ is normal,
  \item $\X$ has affine diagonal,
  \item geometric points of $\X$ have linearly reductive stabilizers, and
  \item $[\X/T]$ is of global type (see Remark \ref{rmk:everything-is-global-type}).
 \end{enumerate}
\end{theorem}
\begin{proof}
 It is clear that any toric stack satisfies the conditions.

 As in the proof of Theorem \ref{thm:main-smooth}, we immediately reduce to the case where $T$ is trivial and $\X$ has a dense open point. By Lemma \ref{lem:check-toric-after-quotient}, it suffices to check that $[\X/T]$ is a toric stack. By Lemma \ref{lem:quot-has-affine-diagonal}, $[\X/T]$ has affine diagonal. By Lemma \ref{lem:quot-has-red-stabilizers}, $[\X/T]$ has linearly reductive stabilizers. Normality and reducedness are local in the smooth topology, so those hypotheses descend from $\X$ to $[\X/T]$.

 Applying Theorem \ref{thm:local-structure}, we obtain an open cover $\bigsqcup \X_i\to \X$, where each $\X_i$ is of the form $[X_i/T_i]$, with $X_i$ an affine toric variety. Let $\Y_i$ be the canonical smooth toric stack over $\X_i$ (see \cite[\S 5]{toricartin1}). Since the maps $\Y_i\to \X_i$ have the universal property in \cite[Proposition 5.5]{toricartin1}, they are canonically isomorphic when pulled back to intersections, so they glue together into a smooth stack $\Y\to \X$.

 The diagonal of $\Y$ is affine by Lemma \ref{lem:canonical-stack-affine-diagonal}, and it satisfies the other hypotheses of Theorem \ref{thm:main-smooth} by construction (they are local conditions which all canonical stacks satisfy), so $\Y$ is a smooth toric stack. So by Theorem \ref{thm:local-gms-data}, $\X$ is a toric stack.
\end{proof}

\begin{remark}\label{rmk:tyomkin}
 As an application of Theorem \ref{thm:main}, we see that the toric stacks defined by Tyomkin in \cite[\S 4]{tyomkin} via gluing are in fact globally quotients of toric varieties by subgroups of their tori. The fact that Tyomkin's stacks are constructed from toric stacky data (\cite[Definition 4.1]{tyomkin}) implies that they have affine diagonal. The other conditions of Theorem \ref{thm:main} are clearly satisfied.
\end{remark}

\begin{remark}[There are no ``toric algebraic spaces'']\label{rmk:no-toric-spaces}
 Suppose $X$ is a toric variety and $G$ is a subgroup of the torus of $X$ such that the toric stack $[X/G]$ is an algebraic space (i.e.~$G$ acts \emph{freely} on $X$). Then $[X/G]$ is a scheme. This can be seen by noting that $X$ is covered by torus invariant (and therefore $G$-invariant) open affines; the stack quotient of an affine scheme by a free action of $G$ is an affine scheme, so $[X/G]$ has a Zariski open cover by affine schemes. Theorem \ref{thm:main} therefore shows that any toric algebraic space satisfying the conditions of the theorem is a scheme.
 
 Any separated toric algebraic space over $\mathbb C$ is a toric variety \cite[Theorem 1]{hausen}. By reducing to $\mathbb C$, Skowera has shown that any separated toric algebraic space over an algebraically closed field of characteristic 0 satisfies condition $4$ (the other conditions are immediate), and therefore that any such algebraic space is a toric variety \cite[Theorem 3.7]{skowera}.
\end{remark}

\begin{example}[A (non-normal) toric algebraic space]\label{eg:a-toric-space}
 If the conditions of Theorem \ref{thm:main} are not imposed, then there are toric algebraic spaces which are not schemes. For example, the ``line with a doubled tangent direction'' \cite[Example 1 in the Introduction]{knutson} is a toric algebraic space. Note that the normalization of this algebraic space is the non-separated line, which \emph{does} arise from a stacky fan \cite[Example 2.12]{toricartin1}.
\end{example}

\section{Application: A Good Moduli Space of a Toric Stack is Toric}

As an application of Theorem \ref{thm:main}, we show that if a toric stack $\X$ has a good moduli space $X$, then $X$ must be toric and the good moduli space morphism $\X\to X$ is toric. This is used in the proof of \cite[Corollary 6.5]{toricartin1}, which gives a combinatorial criterion for when a toric stack has a variety as a good moduli space. See also \cite[Remark 6.7]{toricartin1}, which gives a criterion for when a toric stack has a good moduli space.

\begin{lemma}\label{lem:can-remove-divisors-with-high-codimension-image}
  Suppose $f\colon \X\to \Y$ is a good moduli space morphism from a smooth stack. Let $\Z\subseteq \Y$ be a closed substack of codimension at least $2$, and let $\D\subseteq \X$ be the union of the components of $f^{-1}(\Z)$ which are of codimension $1$ in $\X$. Then the restriction $g\colon \X\setminus \D\to \Y$ is a good moduli space morphism.
\end{lemma}
\begin{proof}
  Every divisor of a smooth stack is Cartier by Lemma \ref{lem:Weil=Cartier}, so the inclusion $\X\setminus \D\to \X$ is cohomologically affine. As cohomologically affine morphisms are stable under composition, $g$ is cohomologically affine.

  To see that $g$ is Stein, consider the following diagram, in which the outer square is cartesian:
  \[\xymatrix{
    \X\setminus f^{-1}(\Z)\ar[d]_{f'} \ar@{^(->}[r]^-j & \X\setminus \D\ar@{^(->}[r]^-{j'} & \X\ar[d]^f\\
    \Y\setminus \Z\ar@{^(->}[rr]^-i && \Y
  }\]
  By assumption, $i$ and $j$ are inclusions of open substacks whose complements have codimension at least 2, so these maps are Stein by Lemma \ref{lem:stein-open-immersion}. Since good moduli space morphisms are stable under base change \cite[Proposition 4.6(i)]{Alper:good}, $f'$ is Stein. We then have that
  \[
    g_*\O_{\X\setminus \D} = f_*j'_*\O_{\X\setminus \D} = f_*j'_*j_*\O_{\X\setminus f^{-1}(\Z)} = i_*f'_*\O_{\X\setminus f^{-1}(\Z)} = \O_\Y.\qedhere
  \]
\end{proof}

\begin{proposition}
\label{prop:toric-alg-sp-gms}
  Suppose $\X$ is a toric stack, $f\colon \X\to Y$ is a good moduli space morphism, and $Y$ is an algebraic space. Then $Y$ is a toric stack (and therefore a scheme by Remark \ref{rmk:no-toric-spaces})\footnote{Note that $Y$ may not be a \emph{variety} since it may be non-separated.} and $f$ is a toric morphism.
\end{proposition}
\begin{proof}
  To prove the result, we may replace $\X$ by its canonical stack, so we may assume $\X$ is smooth. By Lemma \ref{lem:can-remove-divisors-with-high-codimension-image}, we may remove all torus-invariant divisors in $\X$ whose image in $Y$ have codimension larger than 1, so we may assume the image of every torus-invariant divisor of $\X$ is either a divisor or equal to $Y$.

  By \cite[Proposition 4.16(viii)]{Alper:good}, $Y$ is normal. The stabilizers of points of $Y$ are trivial and therefore linearly reductive. By Lemma \ref{l:affine-diag-gms}, $Y$ has affine diagonal. To show $Y$ is toric by Theorem \ref{thm:main}, it remains to show that $Y$ has an action of a dense open torus and that the quotient by that torus is of global type.

  Let $T\subseteq \X$ be the dense torus. Since good moduli space morphisms from locally noetherian stacks are initial among maps to algebraic spaces \cite[Theorem 6.6]{Alper:good}, the action of $T$ on $\X$ induces an action of $T$ on $Y$, making $f$ a $T$-equivariant morphism.

  We induct on the number of $T$-invariant divisors of $\X$. If $\X$ has no $T$-invariant divisors, then $\X=T$ has good moduli space $Y=T$, so $f$ is an isomorphism.

  If there is an irreducible $T$-invariant divisor $\D\subseteq \X$ which dominates $Y$, then by \cite[Lemma 4.14]{Alper:good}, $\D\to Y$ is a good moduli space morphism. By \cite[Proposition 7.20]{toricartin1}, $\D$ is an essentially trivial gerbe over a smooth toric stack $\bbar \D$. We then have the induced good moduli space morphism $\bbar \D\to Y$. Since $\bbar\D$ has fewer torus-invariant divisors than $\X$, we have by the inductive hypothesis that $Y$ is toric and that $\bbar\D\to Y$ is a toric morphism. Since the torus of $\bbar\D$ is a quotient of the torus of $\X$, the map $\X\to Y$ is toric.

  We may therefore assume that the image of every $T$-invariant divisor $\D_i$ of $\X$ is a divisor $D_i$ in $Y$. By induction, the complement of any $D_i$ in $Y$ is toric and the good moduli space morphism from $\X\setminus f^{-1}(D_i)$ is toric. In particular, $Y$ has a dense open torus $T_Y$ whose multiplication extends to an action on $Y\setminus \bigcap D_i$. If $\bigcap D_i=\varnothing$, then $Y$ is Zariski locally a toric stack, so $[Y/T_Y]$ is of global type.

  We may therefore assume that $\bigcap D_i\neq\varnothing$. By \cite[Theorem 4.16(iii)]{Alper:good}, $\bigcap \D_i\neq\varnothing$, so $\X$ is cohomologically affine. Since $\X$ is smooth, we have that $\X=[\AA^n/G]$ for some subgroup $G\subseteq \GG_m^n$. It follows that $Y=\AA^n/G$ is toric and the map $[\AA^n/G]\to \AA^n/G$ is toric.
\end{proof}



\end{document}